\documentclass[11pt]{amsart}
\usepackage{graphicx}
\usepackage[centertags]{amsmath}
\usepackage{amsfonts}
\usepackage{amssymb}
\usepackage{amsthm}
\usepackage{newlfont}
\newtheorem{theorem}{Theorem}[section]
\newtheorem{corollary}[theorem]{Corollary}
\newtheorem{lemma}[theorem]{Lemma}
\newtheorem{proposition}[theorem]{Proposition}
\theoremstyle{definition}
\newtheorem{definition}[theorem]{Definition}
\theoremstyle{remark}
\newtheorem{remark}[theorem]{Remark}
\newcommand{\id}{\textrm{id}}
\newcommand{\grad}{\textrm{grad}}

\title{On 3-dimensional almost Einstein manifolds with circulant structures}

\author{Iva Dokuzova}
\address{Iva Dokuzova\\Department of Algebra and Geometry\\
Plovdiv University Paisii Hilendarski\\ 24 Tzar Asen, 4000 Plovdiv, Bulgaria}
\email{dokuzova@uni-plovdiv.bg}

\begin{document}

\begin{abstract}
 A 3-dimensional Riemannian manifold equipped  with a tensor structure of type $(1,1)$, whose third power is the identity, is considered. This structure and the metric have circulant matrices with respect to some basis, i.e., these structures are circulant. An associated manifold, whose metric is expressed by both structures, is studied. Three classes of such manifolds are considered. Two of them are determined by special properties of the curvature tensor of the manifold. The third class is composed by manifolds whose structure is parallel with respect to the Levi-Civita connection of the metric. Some geometric characteristics of these manifolds are obtained. Examples of such manifolds are given.
\end{abstract}


\maketitle

\section{Introduction}
\label{Sec:1}
Significant results in the geometry of Riemannian manifolds with additional structures are related to the curvature tensor, the Ricci tensor, the scalar curvatures, the Ricci curvature and the sectional curvatures of some characteristic 2-planes of every tangent space of the manifolds.
  We will mention the following papers on this topic. Some of them refer to the theory of Riemannian almost product manifolds (\cite{griba-mek2, Mek2, S-G}), and others refer to the theory of almost Hermitian manifolds (\cite{Sca-Vezz, prvn, van, yu}). There are studied classes of manifolds, whose curvature tensors are invariant under the additional structure, with interesting geometrical characteristics. A.~Naveira made a classification of Riemannian almost product manifolds by the properties of the tensor $\nabla P$, where $\nabla$ is the Levi-Civita connection determined by the metric, and $P$ is the almost product structure (\cite{Nav}). The class $W_{0}$ defined by $\nabla P=0$ in this classification is common to all classes. Every manifold in $W_{0}$ has curvature tensor which is invariant under $P$. In this way, almost Hermitian manifolds were classified by A.~Gray and L.~Hervella (\cite{GrHer}). In \cite{gray}, there are introduced three classes determined by Gray's curvature identities and it is proved that every K\"{a}hler manifold satisfies them. Due to A.~Gray, in these classes curvature identities are a key to understand their geometry.

  We consider a 3-dimensional Riemannian manifold $(M, g, Q)$. Here $g$ is the metric and $Q$ is a tensor field of type $(1,1)$, such that $Q^{3}=\id$, $Q\neq\id$. The local coordinates of $Q$ form a circulant matrix and $Q$ is compatible with $g$, such that an isometry is induced in any tangent space of $M$. Also, we consider an associated manifold $(M,\tilde{g}, Q)$ whose metric $\tilde{g}$ is expressed by $g$ and $Q$, and $\tilde{g}$ is necessarily indefinite. We study two classes $\mathcal{L}_{2}$ and $\mathcal{L}_{1}$ of manifolds whose curvature tensors are invariant under $Q$. The class $\mathcal{L}_{0}$, composed by manifolds whose structure $Q$ is parallel with respect to the Levi-Civita connection of the metric, is their subclass. Our purpose is to obtain some geometric properties of $(M, \tilde{g}, Q)$ and relations between curvature quantities of $(M, g, Q)$ and $(M, \tilde{g}, Q)$, when these manifolds belong to $\mathcal{L}_{2}$, $\mathcal{L}_{1}$, or $\mathcal{L}_{0}$.

The paper is organized as follows. In Section \ref{Sec:2}, we recall some basic facts about $(M, g, Q)$ and $(M, \tilde{g}, Q)$ known from \cite{Raz, DRDok, AE-dokuzova}. In Section \ref{Sec:3}, we obtain conditions for $(M, g, Q)$ which are necessary and sufficient for belonging of $(M, \tilde{g}, Q)$ to each of the classes $\mathcal{L}_{2}$, $\mathcal{L}_{1}$ and $\mathcal{L}_{0}$. In both classes $\mathcal{L}_{2}$ and $\mathcal{L}_{1}$, we express the Ricci tensor of $(M, \tilde{g}, Q)$ by the metrics $g$ and $\tilde{g}$, and establish that $(M, \tilde{g}, Q)$ is an almost Einstein manifold. Also, we get a condition under which the manifold $(M, \tilde{g}, Q)$ is Einstein. In Section \ref{Sec:4}, we obtain the sectional curvatures of some characteristic 2-planes of $(M, \tilde{g}, Q)$. For an Einstein manifold, we find the Ricci curvature in the direction of a non-isotropic vector, as well as in the direction of an isotropic vector. In Section \ref{Sec:5}, we characterize geometrically examples of the considered manifolds on 3-dimensional real Lie groups, which are constructed in \cite{fil-dokuzova}.

\section{Preliminaries}
\label{Sec:2}
We continue our investigations on manifolds $(M, g, Q)$ and $(M, \tilde{g}, Q)$ studied in \cite{Raz, DRDok, AE-dokuzova}. These manifolds are determined in the following way.

Let $M$ be a $3$-dimensional differentiable manifold equipped with a Riemannian metric $g$.
Let $Q$ be a tensor field on $M$ of type $(1,1)$ whose coordinate matrix, with respect to some basis $\{e_{1}, e_{2}, e_{3}\}$ of the tangent space $T_{p}M$, $p\in M$, is a circulant one:
\begin{equation}\label{f4}
    (Q_{i}^{j})=\begin{pmatrix}
      0 & 1 & 0 \\
      0 & 0 & 1 \\
      1 & 0 & 0 \\
    \end{pmatrix}.
\end{equation}
Obviously
\begin{equation}\label{Q3}
    Q^{3}= \id,\qquad Q\neq \id.
\end{equation}
Let the structure $Q$ be compatible with $g$ such that
\begin{equation}\label{2.12}
     g(Qx, Qy)=g(x,y).
\end{equation}
 Here and anywhere in this work, $x, y, z, u$ will stand for arbitrary elements of the algebra on the smooth vector fields on $M$ or vectors in $T_{p}M$. The Einstein summation convention is used, the range of the summation indices being always $\{1, 2, 3\}$.

The equalities \eqref{Q3} and \eqref{2.12} imply that the matrix of $g$ has the form
\begin{equation}\label{f2}
    (g_{ij})=\begin{pmatrix}
      A & B & B \\
      B & A & B \\
      B & B & A \\
    \end{pmatrix},
\end{equation}
 where $A$ and $B$ are smooth functions on $M$. We suppose $A>B>0$ in order that the metric $g$ is positive definite.

The associated metric $\tilde{g}$ on $(M, g, Q)$ is determined by \begin{equation}\label{metricf}
 \tilde{g}(x, y)=g(x, Qy)+g(Qx, y).\end{equation}
 It is an indefinite metric whose component matrix has the form
\begin{equation}\label{f21}
(\tilde{g}_{ij})=\begin{pmatrix}
      2B & A+B & A+B \\
      A+B & 2B & A+B \\
      A+B & A+B & 2B \\
    \end{pmatrix}.
\end{equation}
Further, we will say that $(M, \tilde{g}, Q)$ is associated with $(M, g, Q)$.

\begin{definition}
A basis of type $\{x, Qx, Q^{2}x\}$ of $T_{p}M$ is called a $Q$-\textit{basis}. In this case we say that \textit{the vector $x$ induces a $Q$-basis of} $T_{p}M$.
\end{definition}

In \cite{Raz}, for $(M, g, Q)$ it is verified that:
\begin{itemize}
  \item [(i)] if a vector $x$ induces a $Q$-basis of $T_{p}M$ and $\varphi$ is the angle between $x$ and $Qx$ with respect to $g$, then
   \begin{equation}\label{varphi}
   \varphi\in \big(0,\dfrac{2\pi}{3}\big),\qquad \angle(x, Qx)=\angle(Qx, Q^{2}x)=\angle(x, Q^{2}x)=\varphi;
   \end{equation}
    \item [(ii)] an orthogonal $Q$-basis of $T_{p}M$ exists.
\end{itemize}

The Levi-Civita connection on a Riemannian manifold is denoted by $\nabla$. The curvature tensor $R$ of $\nabla$ is defined by \begin{equation}\label{R-def}R(x, y)z=\nabla_{x}\nabla_{y}z-\nabla_{y}\nabla_{x}z-\nabla_{[x,y]}z.\end{equation}
Also, we consider the tensor of type $(0, 4)$ associated with $R$, defined as follows
\begin{equation*}
    R(x, y, z, u)=g(R(x, y)z,u).
\end{equation*}

 A manifold $(M, g, Q)$ is in class $\mathcal{L}_{0}$ if the structure $Q$ is parallel with respect to $g$, i.e.
 \begin{equation*}
 \nabla Q=0.\end{equation*}

 A  manifold $(M, g, Q)$ is in class $\mathcal{L}_{1}$ if
\begin{equation}\label{V1}
  R(x, y, Qz, Qu)=R(x, y, z, u).
\end{equation}

A  manifold $(M, g, Q)$ is in class $\mathcal{L}_{2}$ if
\begin{equation}\label{V2}
  R(Qx, Qy, Qz, Qu)=R(x, y, z, u).
\end{equation}
The subsets $\mathcal{L}_{0}\subset\mathcal{L}_{1}\subset\mathcal{L}_{2}$ are valid (\cite{AE-dokuzova}).

Let $R_{ijkh}$ be the components of the curvature tensor $R$ of type $(0, 4)$. The following statements are presented in \cite{AE-dokuzova}. \begin{proposition}\label{ae-dok}
The property \eqref{V1} of the manifold $(M, g, Q)$ is equivalent to the conditions \begin{equation}\label{r1=r2}
    R_{1212}=R_{1313}=R_{2323}=- R_{1213}=-R_{1323}=R_{1223}.
\end{equation}
\end{proposition}
\begin{proposition}\label{1}
  The property \eqref{V2} of the manifold $(M, g, Q)$ is equivalent to the conditions
  \begin{equation}\label{r1-r6}
    R_{1212}=R_{1313}=R_{2323},\quad R_{1213}=R_{1323}=-R_{1223}.
\end{equation}
\end{proposition}

 The Ricci tensor $\rho$ and the scalar curvatures $\tau$ and $\tau^{*}$, with respect to $g$, are given by the well-known formulas:
\begin{equation}\label{def-rho}
    \rho(y,z)=g^{ij}R(e_{i}, y, z, e_{j}),\quad
    \tau=g^{ij}\rho(e_{i}, e_{j})\quad \tau^{*}=\tilde{g}^{ij}\rho(e_{i}, e_{j}).
\end{equation}
In \eqref{def-rho} we denote by $g^{ij}$ and $\tilde{g}^{ij}$ the components of the inverse matrices of $(g_{ij})$ and $(\tilde{g}_{ij})$, respectively.

A Riemannian manifold is said to be Einstein if its Ricci tensor $\rho$ is a constant multiple of the metric tensor $g$, i.e.
\begin{equation}\label{E}\rho(x, y) = \alpha g(x, y).\end{equation}

In \cite{Yano}, for locally decomposable Riemannian manifolds is defined a class of almost Einstein manifolds.
For the considered in our paper manifolds, we suggest the following
\begin{definition}\label{defAE}
A Riemannian manifold $(M, g, Q)$ is called
almost Einstein if the metrics $g$ and $\tilde{g}$ satisfy
\begin{equation*}
\rho(x, y) = \alpha g(x, y) + \beta \tilde{g}(x, y),\end{equation*} where $\alpha$ and $\beta$ are smooth functions on $M$.
\end{definition}

\section{Almost Einstein manifolds}\label{Sec:3}

We consider a manifold $(M, g, Q)$ and the associated manifold $(M, \tilde{g}, Q)$.

Let $\tilde{\nabla}$
be the Levi-Civita connection of $\tilde{g}$ and $\tilde{R}$
be the curvature tensor of $\tilde{\nabla}$. The Ricci tensor $\tilde{\rho}$ and the scalar curvatures $\tilde{\tau}$ and $\tilde{\tau}^{*}$, with respect to $\tilde{g}$, are
\begin{equation}\label{def-rho2}
    \tilde{\rho}(y,z)=\tilde{g}^{ij}\tilde{R}(e_{i}, y, z, e_{j}),\quad
    \tilde{\tau}=\tilde{g}^{ij}\tilde{\rho}(e_{i}, e_{j}),\quad \tilde{\tau}^{*}=g^{ij}\tilde{\rho}(e_{i}, e_{j}).
\end{equation}
In \cite{fil-dokuzova}, for $(M, g, Q)$ and $(M, \tilde{g}, Q)$, it is established the following
\begin{theorem}\label{connR-R}
 For the Ricci tensors  $\rho$ and $\tilde{\rho}$ and for the scalar curvatures $\tau$, $\tau^{*}$, $\tilde{\tau}$ and $\tilde{\tau}^{*}$ the following relation is valid:
 \begin{equation}\label{con-AE}
     \tilde{\rho}(x,y) = \rho(x,y)+\frac{1}{3}(\tilde{\tau}^{*}-\tau)g(x,y)+\frac{1}{6}(2\tilde{\tau}-2\tau^{*}+\tilde{\tau}^{*}-\tau)\tilde{g}(x,y).
\end{equation}
\end{theorem}
Further, we apply formulas \eqref{R-def} -- \eqref{r1-r6} to $\tilde{g}$,  $\tilde\nabla$ and $\tilde{R}$.

\subsection{The class $\mathcal{L}_{2}$}

For the manifold $(M, g, Q)$ the following propositions are equivalent (\cite{AE-dokuzova}):
\begin{itemize}
\item[(i)] $(M, g, Q)$ belongs to $\mathcal{L}_{2}$;
  \item[(ii)] the components of the Ricci tensor $\rho$ are
\begin{equation}\label{system-rho}
\rho_{11}=\rho_{22}=\rho_{33},\quad \rho_{12}=\rho_{13}=\rho_{23};
\end{equation}
  \item[(iii)] $(M, g, Q)$ is an almost Einstein manifold and the Ricci tensor $\rho$ is expressed by
 \begin{equation}\label{rho51}
     \rho(x,y)=\frac{\tau}{3}g(x,y)+\big(\frac{\tau}{6}+\frac{\tau^{*}}{3}\big)\tilde{g}(x,y).
\end{equation}
\end{itemize}

Further, we will obtain some geometric properties of $(M, \tilde{g}, Q)$ when $(M, g, Q)$ is in $\mathcal{L}_{2}$. For this purpose, first we state the following
\begin{theorem}\label{thL2-l2}
A manifold $(M, g, Q)$ belongs to $\mathcal{L}_{2}$ if and only if $(M, \tilde{g}, Q)$ belongs to $\mathcal{L}_{2}$.
\end{theorem}
\begin{proof}
 The local form of \eqref{con-AE} is \begin{equation}\label{con-AE-loc}
     \tilde{\rho}_{ij} = \rho_{ij}+\frac{1}{3}(\tilde{\tau}^{*}-\tau)g_{ij}+\frac{1}{6}(2\tilde{\tau}-2\tau^{*}+\tilde{\tau}^{*}-\tau)\tilde{g}_{ij}.
\end{equation}
Let $(M, g, Q)$ belong to $\mathcal{L}_{2}$. From \eqref{con-AE-loc}, having in mind \eqref{f2}, \eqref{f21} and \eqref{system-rho}, we get \begin{equation}\label{system-tilde-rho}
\tilde{\rho}_{11}=\tilde{\rho}_{22}=\tilde{\rho}_{33},\quad \tilde{\rho}_{12}=\tilde{\rho}_{13}=\tilde{\rho}_{23}.
\end{equation}

It is known that the curvature tensor $R$ for a $3$-dimensional Riemannian manifold is completely determined by the Ricci tensor $\rho$ and the metric $g$, as follows
\begin{equation*}
\begin{split}
R(x, y, z, u)=&-g(x, z)\rho(y,u)-g(y, u)\rho(x, z)+g(y, z)\rho(x, u)\\+&g(x, u)\rho(y, z)+\frac{\tau}{2}\big(g(x, z)g(y, u)-g(y, z)g(x, u)\big).
\end{split}
\end{equation*}
The local form of the above identity, written for $\tilde{R}$, is
\begin{equation}\label{loc-R-tilde}
\tilde{R}_{ijkl}= -\tilde{g}_{ik}\tilde{\rho}_{jl}-\tilde{g}_{jl}\tilde{\rho}_{ik}+\tilde{g}_{jk}\tilde{\rho}_{il}+\tilde{g}_{il}\tilde{\rho}_{jk}+\frac{\tilde{\tau}}{2}\big(\tilde{g}_{ik}\tilde{g}_{jl}-\tilde{g}_{jk}\tilde{g}_{il}\big).
\end{equation}
Taking into account \eqref{f21}, \eqref{system-tilde-rho} and \eqref{loc-R-tilde} we find that the components of the curvature tensor $\tilde{R}$ satisfy
\begin{equation}\label{r1-r6-tilde}
    \tilde{R}_{1212}=\tilde{R}_{1313}=\tilde{R}_{2323},\quad \tilde{R}_{1213}=\tilde{R}_{1323}=-\tilde{R}_{1223}.
\end{equation}
According to Proposition~\ref{1}, the equalities \eqref{r1-r6-tilde} imply that $(M, \tilde{g}, Q)$ belongs to $\mathcal{L}_{2}$.

Conversely, for $(M, \tilde{g}, Q)\in\mathcal{L}_{2}$ conditions \eqref{r1-r6-tilde} hold.
On the other hand, bearing in mind \eqref{f21}, we state that the components of the inverse matrix of $\tilde{g}$ satisfy the equalities
  $\tilde{g}^{11}=\tilde{g}^{22}=\tilde{g}^{33}$ and $\tilde{g}^{12}=\tilde{g}^{13}=\tilde{g}^{23}.$
Then, using \eqref{r1-r6-tilde} and the first equality of \eqref{def-rho2}, we get \eqref{system-tilde-rho}. Therefore, from \eqref{f2}, \eqref{f21}, \eqref{con-AE-loc} and \eqref{system-tilde-rho} it follows that that equalities \eqref{system-rho} are valid, i.e. $(M, g, Q)\in\mathcal{L}_{2}$.
\end{proof}
In the course of the above proof we obtain the following
\begin{corollary}\label{cor-ro}
A manifold $(M, \tilde{g}, Q)$ belongs to $\mathcal{L}_{2}$ if and only if the components of the Ricci tensor $\tilde{\rho}$ satisfy the equalities \eqref{system-tilde-rho}.
\end{corollary}
 \begin{theorem}\label{Th-AE}
A manifold $(M, \tilde{g}, Q)$ belongs to $\mathcal{L}_{2}$ if and only if $(M, \tilde{g}, Q)$ is an almost Einstein manifold. Then we have
 \begin{equation}\label{rho51-tilde}
     \tilde{\rho}(x,y)=\frac{\tilde{\tau}^{*}}{3}g(x,y)+\big(\frac{\tilde{\tau}}{3}+\frac{\tilde{\tau}^{*}}{6}\big)\tilde{g}(x,y).
\end{equation}
\end{theorem}
\begin{proof} Let $(M, \tilde{g}, Q)$ belong to $\mathcal{L}_{2}$. Due to Theorem~\ref{thL2-l2}, $(M, g, Q)$ is in $\mathcal{L}_{2}$ and the equality \eqref{rho51} is valid. We substitute \eqref{rho51} into \eqref{con-AE} and obtain \eqref{rho51-tilde}, i.e. $(M, \tilde{g}, Q)$ is an almost Einstein manifold.

Vice versa. Let $(M, \tilde{g}, Q)$ be an almost Einstein manifold. Bearing in mind Definition~\ref{defAE} we have $\tilde{\rho}(x,y)=\alpha g(x,y)+\beta\tilde{g}(x,y)$. Applying \eqref{f2} and \eqref{f21} in the latter expression of $\tilde{\rho}$ we get \eqref{system-tilde-rho}. Due to Corollary~\ref{cor-ro}, the manifold $(M, \tilde{g}, Q)$ is in $\mathcal{L}_{2}$.
\end{proof}

\begin{theorem}\label{Th-RAE}
Let $(M, \tilde{g}, Q)$ belong to $\mathcal{L}_{2}$. Then the curvature tensor $\tilde{R}$ has the form
 \begin{equation}\label{fR-tilde}
\tilde{R} = \big(\frac{\tilde{\tau}^{*}}{3}+\frac{\tilde{\tau}}{6}\big)\tilde{\pi}_{1}+\frac{\tilde{\tau}^{*}}{3}\tilde{\pi}_{2},
\end{equation}
where \begin{align}\label{pi-pi}
\begin{split}
\tilde{\pi}_{1}(x, y, z, u )&=\mbox{}\tilde{g}(y, z)\tilde{g}(x, u) - \tilde{g}(x, z)\tilde{g}(y, u),\\
\tilde{\pi}_{2}(x, y, z, u)&=\mbox{}g(y, z)\tilde{g}(x, u)+g(x, u)\tilde{g}(y, z)\\&-g(x, z)\tilde{g}(y, u)-g(y, u)\tilde{g}(x, z).
\end{split}
\end{align}
\end{theorem}
\begin{proof}
The proof follows directly from \eqref{loc-R-tilde} and \eqref{rho51-tilde}.
\end{proof}

Next, with the help of Theorem~\ref{Th-AE}, we establish the following
\begin{corollary}\label{EM}
Let $(M, \tilde{g}, Q)$ belong to $\mathcal{L}_{2}$. Then  $(M, \tilde{g}, Q)$ is an Einstein manifold if and only if the scalar curvature $\tilde{\tau}^{*}$ is equal to zero.
\end{corollary}
\begin{proof}
The equality \eqref{rho51-tilde} is similar to \eqref{E} if and only if $\tilde{\tau}^{*}=0$. In this case the Ricci tensor has the form \begin{equation}\label{E2}\tilde{\rho}(x,y)=\dfrac{\tilde{\tau}}{3}\tilde{g}(x,y).\end{equation}
\end{proof}
Now, taking into account \eqref{f21}, \eqref{system-tilde-rho} and \eqref{E2}, we make the following
\begin{remark} Every Einstein manifold $(M, \tilde{g}, Q)$ belongs to $\mathcal{L}_{2}$.\end{remark}
\subsection{The class $\mathcal{L}_{1}$}

\begin{theorem}\label{prop-l1}
Let $(M, \tilde{g}, Q)$ belong to $\mathcal{L}_{2}$. Then $(M, \tilde{g}, Q)$ belongs to $\mathcal{L}_{1}$ if and only if the scalar curvatures satisfy $\tilde{\tau}^{*}=-\tilde{\tau}$.
\end{theorem}
\begin{proof}
 From \eqref{f2}, \eqref{f21}, \eqref{r1-r6-tilde}, \eqref{fR-tilde} and \eqref{pi-pi} we calculate the components of $\tilde{R}$:
\begin{equation}\label{r-a-b}
\begin{split}
\tilde{R}_{1212} &= \tilde{R}_{1313}=\tilde{R}_{2323}=\frac{\tilde{\tau}^{*}}{3}(A^{2}-B^{2})+\frac{\tilde{\tau}}{6}(A^{2}+2AB-3B^{2}),\\
\tilde{R}_{1213} &= \tilde{R}_{1323} = -\tilde{R}_{1223} = \frac{\tilde{\tau}^{*}}{3}(AB-B^{2})+\frac{\tilde{\tau}}{6}(A^{2}-B^{2}).
\end{split}
\end{equation}

If $(M, \tilde{g}, Q)\in\mathcal{L}_{1}$ then, due to Proposition~\ref{ae-dok}, we have $\tilde{R}_{1212}=-\tilde{R}_{1213}$. Hence, from \eqref{r-a-b} it follows that
\begin{equation*}
(\tilde{\tau}-\tilde{\tau}^{*})(A-B)(A+2B)=0.
\end{equation*}
Therefore, having in mind that $A>B>0$, we get $\tilde{\tau}^{*}=-\tilde{\tau}$.

Vice versa. Let $\tilde{\tau}^{*}=-\tilde{\tau}$ be valid.  From \eqref{r-a-b} we find that $\tilde{R}_{1212}=-\tilde{R}_{1213}$, so $(M, \tilde{g}, Q)$ is in $\mathcal{L}_{1}$.
\end{proof}
Furthermore, due to Theorem~\ref{Th-AE}, Theorem~\ref{Th-RAE} and Theorem~\ref{prop-l1}, we obtain
\begin{corollary}\label{l1-rho}
Let $(M, \tilde{g}, Q)$ belong to $\mathcal{L}_{1}$. Then
\begin{itemize}
  \item[(i)] the Ricci tensor is $\tilde{\rho}(x,y)=\dfrac{\tilde{\tau}}{6}\big(\tilde{g}(x,y)-2g(x,y)\big)$ and it is degenerate;
  \item[(ii)] the curvature tensor is
     $\tilde{R} = -\dfrac{\tilde{\tau}}{6}(\tilde{\pi}_{1}+2\tilde{\pi}_{2}).$
\end{itemize}
 \end{corollary}

\begin{proof}
(i) Applying the equality $\tilde{\tau}^{*}=-\tilde{\tau}$ in \eqref{rho51-tilde} we get $$\tilde{\rho}=\dfrac{\tilde{\tau}}{6}\big(\tilde{g}(x,y)-2g(x,y)\big),$$ which because of \eqref{f2} and \eqref{f21} yields $$\tilde{\rho}_{11}=\dfrac{\tilde{\tau}}{3}(B-A),\ \tilde{\rho}_{12}=\dfrac{\tilde{\tau}}{6}(A-B).$$ Consequently, we have that $\det(\tilde{\rho}_{ij})=0$.

(ii) Using the equality $\tilde{\tau}^{*}=-\tilde{\tau}$, from \eqref{fR-tilde} we find $\tilde{R} = -\dfrac{\tilde{\tau}}{6}(\tilde{\pi}_{1}+2\tilde{\pi}_{2}),$ where $\tilde{\pi}_{1}$ and $\tilde{\pi}_{2}$ are determined by \eqref{pi-pi}.
\end{proof}
\begin{remark} A manifold $(M, \tilde{g}, Q)\in\mathcal{L}_{1}$ does not admit Einstein metric.\end{remark}

\subsection{The class $\mathcal{L}_{0}$}
In \cite{Raz}, it is proved that $(M, g, Q)$ belongs to $\mathcal{L}_{0}$ if and only if the functions $A$ and $B$ satisfy the following matrix equality
\begin{equation}\label{parallel2}
    \grad A=\grad B\begin{pmatrix}
      -1 & 1 & 1 \\
      1 & -1 & 1 \\
      1 & 1 & -1 \\
    \end{pmatrix}.
\end{equation}
\begin{proposition}\label{confequiv}
The manifold $(M, \tilde{g}, Q)$ belongs to $\mathcal{L}_{0}$ if and only if $(M, g, Q)$ belongs to $\mathcal{L}_{0}$.
\end{proposition}
\begin{proof}
Let $(M, \tilde{g}, Q)$ be in $\mathcal{L}_{0}$. Thus the components of the metric \eqref{f21} satisfy an equality of the type \eqref{parallel2}. Therefore we have
\begin{equation*}
    \grad 2B=\grad (A+B)\begin{pmatrix}
      -1 & 1 & 1 \\
      1 & -1 & 1 \\
      1 & 1 & -1 \\
    \end{pmatrix},
\end{equation*}
i.e.
\begin{equation}\label{sys-g-tilde}
\begin{split} 2B_{1}=-A_{1}-B_{1}+A_{2}+B_{2}+A_{3}+B_{3},\\ 2B_{2}=-A_{2}-B_{2}+A_{3}+B_{3}+A_{1}+B_{1},\\
2B_{3}=-A_{3}-B_{3}+A_{2}+B_{2}+A_{1}+B_{1},\end{split}\end{equation}
where $A_{i}=\frac{\partial A}{\partial x_{i}}$,  $B_{j}=\frac{\partial B}{\partial x_{j}}$.
The system \eqref{sys-g-tilde} is reduced to \begin{equation}\label{paralel3} A_{1}=-B_{1}+B_{2}+B_{3},\quad
A_{2}=B_{1}-B_{2}+B_{3},\quad A_{3}=B_{1}+B_{2}-B_{3},\end{equation}
 which is equivalent to the matrix equality \eqref{parallel2}, i.e., $(M, g, Q)\in\mathcal{L}_{0}$.

Vice versa. If $(M, g, Q)$ is in $\mathcal{L}_{0}$, then we have \eqref{paralel3} which implies \eqref{sys-g-tilde}. Consequently, the components of the metric $\tilde{g}$ satisfy an equality of the type \eqref{parallel2}. Hence $(M, \tilde{g}, Q)\in\mathcal{L}_{0}$.
\end{proof}

\section{Some curvature properties of the considered manifolds}\label{Sec:4}
In this section we consider some curvature properties of the associated manifold $(M, \tilde{g}, Q)$.

According to the well-known definitions we say that:
\begin{itemize}
  \item[(i)] a 2-plane $\{x, y\}$ spanned by the vectors $x, y \in T_{p}M$ is non-degenerate if $\tilde{g}(x, x)\tilde{g}(y, y)-\tilde{g}^{2}(x, y)\neq 0$;
  \item[(ii)] the sectional curvature of a non-degenerate $2$-plane $\{x, y\}$ spanned by the vectors $x, y \in T_{p}M$ is the value
\begin{equation}\label{3.3}
   \tilde{k}(x,y)=\frac{\tilde{R}(x, y, x, y)}{\tilde{g}(x, x)\tilde{g}(y, y)-\tilde{g}^{2}(x, y)}.
\end{equation}
\end{itemize}

Further in this section, we suppose that $x$ induces a $Q$-basis in $T_{p}M$ and $\varphi$ is the angle between $x$ and $Qx$ with respect to $g$. Thus the properties \eqref{varphi} are valid.
\begin{lemma}\label{lem2}
The 2-plane $\{x, Qx\}$ is non-degenerate, with respect to $\tilde{g}$, if and only if $\varphi\neq\arccos\big(-\frac{1}{3}\big)$. If $\{x, Qx\}$ is a non-degenerate 2-plane, then $\{x, Q^{2}x\}$ and $\{Qx, Q^{2}x\}$ are also non-degenerate 2-planes.
\end{lemma}
\begin{proof}
The vectors $x$, $Qx$ and $Q^{2}x$ determine 2-planes $\{x, Qx\}$, $\{x, Q^{2}x\}$ and $\{Qx, Q^{2}x\}$.
With the help of \eqref{2.12}, \eqref{metricf} and \eqref{varphi} we calculate
\begin{equation}\label{g-cos}
\begin{split}
  \tilde{g}(x, Qx)&=\tilde{g}(x, Q^{2}x)=\tilde{g}(Qx, Q^{2}x)=g(x, x)(\cos\varphi+1),\\
  \tilde{g}(x, x)&=\tilde{g}(Qx, Qx)=\tilde{g}(Q^{2}x, Q^{2}x)=2g(x, x)\cos\varphi .
  \end{split}
\end{equation}
Then we obtain that $\{x, Qx\}$ is a non-degenerate 2-plane when the inequality
$$\tilde{g}^{2}(x, x)-\tilde{g}^{2}(x, Qx)=(\cos\varphi-1)(3\cos\varphi+1)g^{2}(x,x)\neq 0$$ holds.
Consequently, the 2-plane $\{x, Qx\}$ is non-degenerate if and only if $\varphi\neq\arccos\big(-\frac{1}{3}\big)$.
Due to \eqref{g-cos}, if $\{x, Qx\}$ is a non-degenerate 2-plane, then $\{x, Q^{2}x\}$ and $\{Qx, Q^{2}x\}$ are also non-degenerate 2-planes.
\end{proof}

\begin{theorem}\label{t2.5}
If $(M, \tilde{g}, Q)$ belongs to $\mathcal{L}_{2}$, then the sectional curvatures of the basic 2-planes are
\begin{equation}\label{obmu}
    \tilde{k}(x,Qx)=\tilde{k}(x,Q^{2}x)= \tilde{k}(Qx,Q^{2}x)=-\frac{\tilde{\tau}^{*}(1+\cos\varphi)}{3(1+3\cos\varphi)}-\frac{\tilde{\tau}}{6},
\end{equation}
where $\varphi\neq \arccos\big(-\frac{1}{3}\big)$.
\end{theorem}
\begin{proof}
Since $(M, \tilde{g}, Q)$ is in $\mathcal{L}_{2}$ we have \eqref{r1-r6}, which implies
\begin{equation}\label{R-glob}
\tilde{R}(x,Qx,x,Qx)=\tilde{R}(x,Q^{2}x,x,Q^{2}x)=\tilde{R}(Qx,Q^{2}x,Qx,Q^{2}x).
\end{equation}
On the other hand, taking into account \eqref{fR-tilde}, \eqref{pi-pi} and \eqref{g-cos}, we find
\begin{equation}\label{Rxqx}
\begin{split}
  \tilde{R}(x,Qx,x,Qx)=&\frac{\tilde{\tau}^{*}}{3}(1-\cos^{2}\varphi)g^{2}(x,x)\\+&\frac{\tilde{\tau}}{6}(1-\cos\varphi)(1+3\cos\varphi)g^{2}(x,x).
  \end{split}
\end{equation}
We apply Lemma~\ref{lem2} and the equalities \eqref{g-cos}, \eqref{R-glob}, \eqref{Rxqx} in \eqref{3.3} and we get \eqref{obmu}.
\end{proof}
\begin{corollary}
If a vector $x$ induces an orthonormal $Q$-basis, then
\begin{equation*}
  \tilde{k}(x,Qx)=\tilde{k}(x,Q^{2}x)= \tilde{k}(Qx,Q^{2}x)=-\frac{\tilde{\tau}^{*}}{3}-\frac{\tilde{\tau}}{6}.
\end{equation*}
\end{corollary}
 The following statement is inspired by results for curvatures of degenerate 2-planes obtained in \cite{d-n}.
\begin{proposition}
Let $(M, \tilde{g}, Q)$ belong to $\mathcal{L}_{2}$ and let $\{x,Qx\}$ be a degenerate 2-plane. Then $\tilde{R}(x, Qx, x, Qx)$ vanishes if and only if $(M, \tilde{g}, Q)$ is an Einstein manifold.
\end{proposition}
\begin{proof}
Let $\{x, Qx\}$ be a degenerate 2-plane. From Lemma~\ref{lem2} it follows that $\varphi= \arccos\big(-\frac{1}{3}\big)$. Thus the equality \eqref{Rxqx} takes the form $\tilde{R}(x,Qx,x, Qx)=\dfrac{8\tilde{\tau}^{*}}{27}g^{2}(x,x)$.
 Then $\tilde{R}(x,Qx,x, Qx)=0$ if and only if $\tilde{\tau}^{*}=0$. Due to Corollary~\ref{EM}, $(M, \tilde{g}, Q)$ is an Einstein manifold.
 \end{proof}
 In case that $(M, \tilde{g}, Q)$ is an Einstein manifold we suggest the following definition:
\textit{The sectional curvature of a degenerate $2$-plane $\{x, Qx\}$ is}
\begin{equation}\label{limk}
   \tilde{k}(x,Qx)=\lim_{\varphi\rightarrow\arccos\big(-\frac{1}{3}\big)}\frac{\tilde{R}(x, Qx, x, Qx)}{\tilde{g}(x, x)\tilde{g}(Qx, Qx)-\tilde{g}^{2}(x, Qx)}.
   \end{equation}
Therefore, we establish the following
\begin{theorem}
Let $(M, \tilde{g}, Q)$ be an Einstein manifold. Then the sectional curvature of a 2-plane $\{x, Qx\}$ is a constant and
\begin{equation}\label{AE-k}
  \tilde{k}(x,Qx)=-\frac{\tilde{\tau}}{6}.
\end{equation}
\end{theorem}
\begin{proof}
Due to Corollary~\ref{EM} we have $\tilde{\tau}^{*}=0$. Hence \eqref{Rxqx} becomes
\begin{equation*}
  \tilde{R}(x,Qx,x,Qx)=\frac{\tilde{\tau}}{6}(1-\cos\varphi)(1+3\cos\varphi)g^{2}(x,x).
\end{equation*}
If $\varphi= \arccos\big(-\frac{1}{3}\big)$ then from the above equality, \eqref{g-cos} and \eqref{limk} we find
\begin{equation*}
\tilde{k}(x,Qx)=\lim_{\varphi\rightarrow\arccos\big(-\frac{1}{3}\big)}\frac{-\tilde{\tau}(1+3\cos\varphi)}{6(1+3\cos\varphi)}=-\frac{\tilde{\tau}}{6}.
\end{equation*}
 If $\varphi\neq \arccos\big(-\frac{1}{3}\big)$ then \eqref{obmu} holds. Hereof the equality
  $\tilde{\tau}^{*}=0$ implies \eqref{AE-k}, which completes the proof.
\end{proof}
With the help of Theorem~\ref{prop-l1} and  Theorem~\ref{t2.5} we state the following
\begin{corollary}\label{muzav1}
If $(M, \tilde{g}, Q)$ belongs to $\mathcal{L}_{1}$, then the sectional curvature of a 2-plane $\{x, Qx\}$ is
\begin{equation*}
    \tilde{k}(x,Qx)=\frac{\tilde{\tau}(1-\cos\varphi)}{6(1+3\cos\varphi)},\end{equation*}
where $\varphi\neq \arccos\big(-\frac{1}{3}\big).$
\end{corollary}

Since $\tilde{g}$ is an indefinite metric it admits isotropic vectors. In \cite{dzhe}, it is established that $x$ is an isotropic vector with respect to $\tilde{g}$ if and only if $\varphi=\frac{\pi}{2}$. Also, if $x$ is an isotropic vector, then $Qx$ and $Q^{2}x$ are isotropic vectors, too.

The Ricci curvature in the direction of a non-isotropic vector $x$ is the value \begin{equation}\label{Ricicurv}
    \tilde{r}(x)=\frac{\tilde{\rho}(x,x)}{\tilde{g}(x,x)}.
\end{equation}
\begin{theorem}\label{ric-l2}
If $(M, \tilde{g}, Q)$ belongs to $\mathcal{L}_{2}$ and $x$ is a non-isotropic vector, then the Ricci curvatures are
\begin{equation}\label{ricc}
     \tilde{r}(x)=\tilde{r}(Qx)=\tilde{r}(Q^{2}x)=\frac{\tilde{\tau}^{*}}{6\cos\varphi}+\big(\frac{\tilde{\tau}^{*}}{6}+\frac{\tilde{\tau}}{3}\big).
\end{equation}
\end{theorem}
\begin{proof}
According to Corollary~\ref{cor-ro} the components of $\tilde{\rho}$ satisfy \eqref{system-tilde-rho}, which implies $\tilde{\rho}(x,x)=\tilde{\rho}(Qx,Qx)=\tilde{\rho}(Q^{2}x,Q^{2}x)$. Hence, taking into account \eqref{rho51-tilde}, we find
\begin{equation}\label{rhoX}
  \tilde{\rho}(x,x)=\tilde{\rho}(Qx,Qx)=\tilde{\rho}(Q^{2}x,Q^{2}x)=\frac{\tilde{\tau}^{*}}{3}g(x,x)+\big(\frac{\tilde{\tau}^{*}}{6}+\frac{\tilde{\tau}}{3}\big) \tilde{g}(x,x).
\end{equation}
Applying  \eqref{g-cos} and \eqref{rhoX} in \eqref{Ricicurv}, we obtain \eqref{ricc}.
\end{proof}
\begin{proposition}
Let $(M, \tilde{g}, Q)$ belong to $\mathcal{L}_{2}$ and let $x$ be an isotropic vector. Then $\tilde{\rho}(x,x)$ vanishes if and only if $(M, \tilde{g}, Q)$ is an Einstein manifold.
\end{proposition}
\begin{proof}
Since $x$ is an isotropic vector it follows that $\varphi=\frac{\pi}{2}$.
 The latter equality, \eqref{g-cos} and \eqref{rhoX} imply $\tilde{\rho}(x,x)=\dfrac{\tilde{\tau}^{*}}{3}g(x,x)$.
 Thus $\tilde{\rho}(x,x)=0$ if and only if $\tilde{\tau}^{*}=0$. Hereof, due to Corollary~\ref{EM}, $(M, \tilde{g}, Q)$ is an Einstein manifold.
 \end{proof}
 In case that $(M, \tilde{g}, Q)$ is an Einstein manifold we suggest the following definition:
 \textit{The Ricci curvature in the direction of an isotropic vector $x$ is}
\begin{equation}\label{limr}
   \tilde{r}(x)=\lim_{\varphi\rightarrow\frac{\pi}{2}}\frac{\tilde{\rho}(x, x)}{\tilde{g}(x, x)}.
   \end{equation}
Therefore, we establish the following
 \begin{theorem}
Let $(M, \tilde{g}, Q)$ be an Einstein manifold. Then the Ricci curvature is a constant and \begin{equation}\label{rho-isotr}
     \tilde{r}(x)=\frac{\tilde{\tau}}{3}.
\end{equation}
\end{theorem}
\begin{proof}
Due to Corollary~\ref{EM} we have $\tilde{\tau}^{*}=0$, and the equality \eqref{rhoX} becomes
\begin{equation}\label{roxx}
\tilde{\rho}(x,x)=\frac{\tilde{\tau}}{3}\tilde{g}(x,x).
\end{equation}
Let $x$ be an isotropic vector, i.e.,  $\varphi=\frac{\pi}{2}$. Therefore, using \eqref{g-cos}, \eqref{limr} and \eqref{roxx}, we calculate
\begin{equation*}
\tilde{r}(x)=\lim_{\varphi\rightarrow\frac{\pi}{2}}\frac{\tilde{\tau}\cos\varphi}{3\cos\varphi}=\frac{\tilde{\tau}}{3}.
\end{equation*}
Let $x$ be a non-isotropic vector. Substituting $\tilde{\tau}^{*}=0$ into \eqref{ricc} we get \eqref{rho-isotr},
which completes the proof.
\end{proof}

Having in mind Theorem~\ref{prop-l1} and Theorem~\ref{ric-l2} we state the following
\begin{corollary}
If $(M, \tilde{g}, Q)$ belongs to $\mathcal{L}_{1}$ and $x$ is a non-isotropic vector, then the Ricci curvature is
\begin{equation*}
     \tilde{r}(x)=\frac{\tilde{\tau}}{6}\big(1-\frac{1}{\cos\varphi}\big).
\end{equation*}
\end{corollary}

\section{Lie groups as manifolds of the considered type}\label{Sec:5}

Let $G$ be a $3$-dimensional real connected Lie group and $\mathfrak{g}$ be its Lie algebra with a basis $\{x_{1}, x_{2},x_{3}\}$ of left invariant vector fields.  The manifold $(G, g, Q)$ equipped with a circulant structure $Q$ and a Riemannian metric $g$, determined by
\begin{equation}\label{lie}
  Qx_{1}=x_{2},\ Qx_{2}=x_{3},\ Qx_{3}=x_{1},
\end{equation}
  \begin{equation}\label{g}
 g(x_{i}, x_{j})= \left\{ \begin{array}{ll}
                        0, & i\neq j \hbox{;} \\
                        1, & i=j \hbox{,}
                      \end{array}
                    \right.
  \end{equation}
 is a manifold of the same type as $(M, g, Q)$ (\cite{fil-dokuzova}).

For the associated metric $\tilde{g}$, using \eqref{metricf}, \eqref{lie} and \eqref{g}, we get
\begin{equation}\label{f}
 \tilde{g}(x_{i}, x_{j})= \left\{ \begin{array}{ll}
                        1, & i\neq j \hbox{;} \\
                        0, & i=j \hbox{.}
                      \end{array}
                    \right.
  \end{equation}
Obviously $(G, \tilde{g}, Q)$ is a manifold of the same type as $(M, \tilde{g}, Q)$.

Let $G'$ be a subgroup of $G$, and $(G', g, Q)$ be a manifold with the curvature tensor which is invariant under $Q$. According to \cite{fil-dokuzova}, we have three classes of Lie algebras $\mathfrak{g}$ satisfying this condition. We consider two of them whose Lie brackets are determined as follows:
 \begin{equation}\label{skobki-ex1}
\begin{split}
  [x_{1}, x_{2}]=&\lambda_{1}x_{1}+\lambda_{2}x_{2},\quad
  [x_{2}, x_{3}]=\lambda_{3}x_{2}-\lambda_{1}x_{3},\\
   [x_{1}, x_{3}]=&\lambda_{3}x_{1}+\lambda_{2}x_{3}.
  \end{split}
\end{equation}
\begin{equation}\label{skobki-ex2}
  [x_{1}, x_{2}]=[x_{2}, x_{3}]=-[x_{1}, x_{3}]=\lambda_{1}x_{1}+\lambda_{2}x_{2}-(\lambda_{1}+\lambda_{2})x_{3}.
\end{equation}

\subsection{Einstein manifolds}

Let $(G', g, Q)$ be a manifold with Lie algebra determined by \eqref{skobki-ex1}. In this case $(G', g, Q)$ belongs to $\mathcal{L}_{2}$ and it is an Einstein manifold (\cite{fil-dokuzova}). Further, we study the associated manifold $(G', \tilde{g}, Q)$.

The well-known Koszul formula implies
\begin{equation}\label{kosh}
    2\tilde{g}(\tilde{\nabla}_{x_{i}}x_{j}, x_{k})=\tilde{g}([x_{i}, x_{j}],x_{k})+\tilde{g}([x_{k}, x_{i}],x_{j})+\tilde{g}([x_{k}, x_{j}],x_{i}).
\end{equation}
Now, using \eqref{f}, \eqref{skobki-ex1} and \eqref{kosh}, we calculate
\begin{align}\label{nabla}\nonumber
    \tilde{\nabla}_{x_{1}}x_{1}&=-\lambda_{2}x_{1},\qquad  \tilde{\nabla}_{x_{2}}x_{2}=\lambda_{1}x_{2},\qquad\tilde{\nabla}_{x_{3}}x_{3}=\lambda_{3}x_{3},\\\nonumber
    2\tilde{\nabla}_{x_{1}}x_{2}&=(\lambda_{1}-\lambda_{2}-\lambda_{3})x_{1}+(\lambda_{1}+\lambda_{2}-\lambda_{3})x_{2}-(\lambda_{1}-\lambda_{2}-\lambda_{3})x_{3},\\\nonumber
   2\tilde{\nabla}_{x_{1}}x_{3}&=-(\lambda_{1}+\lambda_{2}-\lambda_{3})x_{1}+(\lambda_{1}+\lambda_{2}-\lambda_{3})x_{2}-(\lambda_{1}-\lambda_{2}-\lambda_{3})x_{3},\\\nonumber
    2\tilde{\nabla}_{x_{2}}x_{1}&=-(\lambda_{1}+\lambda_{2}+\lambda_{3})x_{1}+(\lambda_{1}-\lambda_{2}-\lambda_{3})x_{2}-(\lambda_{1}-\lambda_{2}-\lambda_{3})x_{3},\\\nonumber
    2\tilde{\nabla}_{x_{2}}x_{3}&=-(\lambda_{1}+\lambda_{2}+\lambda_{3})x_{1}+(\lambda_{1}+\lambda_{2}+\lambda_{3})x_{2}-(\lambda_{1}-\lambda_{2}-\lambda_{3})x_{3},\\\nonumber
    2\tilde{\nabla}_{x_{3}}x_{1}&=-(\lambda_{1}+\lambda_{2}+\lambda_{3})x_{1}+(\lambda_{1}+\lambda_{2}-\lambda_{3})x_{2}-(\lambda_{1}+\lambda_{2}-\lambda_{3})x_{3},\\\nonumber
    2\tilde{\nabla}_{x_{3}}x_{2}&=-(\lambda_{1}+\lambda_{2}+\lambda_{3})x_{1}+(\lambda_{1}+\lambda_{2}-\lambda_{3})x_{2}+(\lambda_{1}+\lambda_{2}+\lambda_{3})x_{3}.
\end{align}
Then, from \eqref{R-def}, \eqref{lie} and \eqref{f}, we obtain all nonzero components of $\tilde{R}$ on
$(G',\tilde{g}, Q)$:
\begin{equation}\label{R-ex1-tilde}
\begin{split}
    \tilde{R}_{1212}=& \tilde{R}_{2323}=\tilde{R}_{1313}=\tilde{R}_{2321}=\tilde{R}_{1213}=\tilde{R}_{1323}\\=&\frac{1}{2}(\lambda^{2}_{1}+\lambda^{2}_{2}+\lambda^{2}_{3})+\lambda_{1}\lambda_{2}+\lambda_{2}\lambda_{3}-\lambda_{1}\lambda_{3}.
     \end{split}
\end{equation}
The latter equalities imply that conditions \eqref{r1-r6} are satisfied, but \eqref{r1=r2} are not satisfied.

Further, from \eqref{def-rho2}, \eqref{g}, \eqref{f} and \eqref{R-ex1-tilde}, we find all nonzero components of $\tilde{\rho}$ and the scalar curvatures of $(G',\tilde{g}, Q)$:
\begin{equation}\label{ro-ex1-tilde}
\begin{split}
    \tilde{ \rho}_{12}= \tilde{\rho}_{13}=\tilde{\rho}_{23}=(\lambda^{2}_{1}+\lambda^{2}_{2}+\lambda^{2}_{3})+2\lambda_{1}\lambda_{2}+2\lambda_{2}\lambda_{3}-2\lambda_{1}\lambda_{3},
     \end{split}
\end{equation}
\begin{equation}\label{tau-ex1-tilde}
    \tilde{\tau}=3(\lambda^{2}_{1}+\lambda^{2}_{2}+\lambda^{2}_{3})+6\lambda_{1}\lambda_{2}+6\lambda_{2}\lambda_{3}-6\lambda_{1}\lambda_{3},\qquad \tilde{\tau}^{*}=0.
\end{equation}
Moreover, taking into account \eqref{f}, \eqref{ro-ex1-tilde} and \eqref{tau-ex1-tilde}, it follows that
$\tilde{\rho}=\dfrac{\tilde{\tau}}{3}\tilde{g},$
 i.e. $(G', \tilde{g}, Q)$ is an Einstein manifold.

We apply \eqref{f} and \eqref{R-ex1-tilde} in \eqref{3.3} and we state that the sectional curvatures $\tilde{k}_{ij}$ of the basic $2$-planes $\{x_{i},  x_{j}\}$ are equal to
\begin{equation}\label{kappa1-tilde}
  \tilde{k}=-\frac{1}{2}(\lambda^{2}_{1}+\lambda^{2}_{2}+\lambda^{2}_{3})-\lambda_{1}\lambda_{2}-\lambda_{2}\lambda_{3}+\lambda_{1}\lambda_{3}.
\end{equation}
Therefore, we establish the truthfulness of the following
\begin{proposition}
 Let $(G', \tilde{g}, Q)$ be a manifold with Lie algebra determined by \eqref{skobki-ex1}. Then the following properties hold:
\begin{itemize}
\item[(i)]  $(G', \tilde{g}, Q)$ belongs to $\mathcal{L}_{2}$ but $(G', \tilde{g}, Q)$ does not belong to $\mathcal{L}_{1}$;
\item[(ii)]  the nonzero components of $\tilde{R}$ and $\tilde{\rho}$ are \eqref{R-ex1-tilde} and \eqref{ro-ex1-tilde}, respectively;
 \item[(iii)]  $(G', \tilde{g}, Q)$ is an Einstein manifold and its scalar curvatures are \eqref{tau-ex1-tilde};
\item[(iv)] $(G', \tilde{g}, Q)$ is of constant sectional curvatures \eqref{kappa1-tilde}.
\end{itemize}
\end{proposition}

\subsection{Almost Einstein manifolds}

Let $(G', g, Q)$ be a manifold with Lie algebra determined by \eqref{skobki-ex2}. In this case $(G', g, Q)$ belongs to $\mathcal{L}_{0}$ and it is an almost Einstein manifold (\cite{fil-dokuzova}).

Now, we consider the associated manifold $(G', \tilde{g}, Q)$.
From \eqref{f}, \eqref{skobki-ex2} and \eqref{kosh} we obtain
\begin{equation}\label{nabla}
\begin{array}{ll}
    \tilde{\nabla}_{x_{1}}x_{1}=\lambda_{1}(x_{3}-x_{2}), &
    \tilde{\nabla}_{x_{1}}x_{2}=\lambda_{1}(x_{1}-x_{3}),\\
   \tilde{\nabla}_{x_{1}}x_{3}=\lambda_{1}(x_{2}-x_{1}), &
    \tilde{\nabla}_{x_{2}}x_{1}=\lambda_{2}(x_{3}-x_{2}),\\ \tilde{\nabla}_{x_{2}}x_{2}=\lambda_{2}(x_{1}-x_{3}),&
    \tilde{\nabla}_{x_{2}}x_{3}=\lambda_{2}(x_{2}-x_{1}),\\
    \tilde{\nabla}_{x_{3}}x_{1}=(\lambda_{1}+\lambda_{2})(x_{2}-x_{3}), &
    \nabla_{x_{3}}x_{2}=(\lambda_{1}+\lambda_{2})(x_{3}-x_{1}),\\ \tilde{\nabla}_{x_{3}}x_{3}=(\lambda_{1}+\lambda_{2})(x_{1}-x_{2}).
    \end{array}
\end{equation}
In the well-known formula $(\tilde{\nabla}_{x_{i}}Q)x_{j}=\tilde{\nabla}_{x_{i}}(Qx_{j})-Q\tilde{\nabla}_{x_{i}}x_{j}$
we apply \eqref{lie} and \eqref{nabla}. Thus we find $\tilde{\nabla}Q=0$, i.e., $(G', \tilde{g}, Q)\in\mathcal{L}_{0}$.

By using \eqref{R-def}, \eqref{def-rho2}, \eqref{lie}, \eqref{f} and \eqref{nabla}  we calculate the components of $\tilde{R}$ and $\tilde{\rho}$:
\begin{align}\label{Rlie2}
    \tilde{R}_{1212}=& \tilde{R}_{2323}=\tilde{R}_{1313}=-\tilde{R}_{1213}=-\tilde{R}_{2123}=-\tilde{R}_{1323}\\\nonumber =&-2(\lambda^{2}_{1}+\lambda^{2}_{2}+\lambda_{1}\lambda_{2}),\end{align}
    \begin{align}\label{rho-lie2}
    \tilde{\rho}_{11}= \tilde{\rho}_{22}=\tilde{\rho}_{33}= -2\tilde{\rho}_{12}= -2\tilde{\rho}_{13}=-2\tilde{\rho}_{23}= -4(\lambda^{2}_{1}+\lambda^{2}_{2}+\lambda_{1}\lambda_{2}).
\end{align}
We find the scalar curvatures and the sectional curvatures of $(G', \tilde{g}, Q)$ with the help of \eqref{def-rho2}, \eqref{3.3}, \eqref{g}, \eqref{f}, \eqref{Rlie2} and \eqref{rho-lie2}.
The obtained results we expose in the following
\begin{proposition} Let $(G', \tilde{g}, Q)$ be a manifold with Lie algebra determined by \eqref{skobki-ex2}. Then the following properties hold:
\begin{itemize}
\item[(i)]  $(G', \tilde{g}, Q)$ belongs to $\mathcal{L}_{0}$;
\item[(ii)] the nonzero components of  $\tilde{R}$ and $\tilde{\rho}$ are \eqref{Rlie2} and \eqref{rho-lie2}, respectively;
 \item[(iii)]  $(G', \tilde{g}, Q)$ is an almost Einstein manifold and its scalar curvatures are
    $\tilde{\tau}=-\tilde{\tau}^{*}=12(\lambda^{2}_{1}+\lambda^{2}_{2}+\lambda_{1}\lambda_{2})$;
\item[(iv)] the sectional curvatures of the basic 2-planes $\{x_{i},  x_{j}\}$ are equal to
  $\tilde{k}=2(\lambda^{2}_{1}+\lambda^{2}_{2}+\lambda_{1}\lambda_{2}).$
\end{itemize}
\end{proposition}


\end{document}